\documentclass[12pt, reqno, a4paper]{amsart}

\usepackage{amssymb}
\usepackage{hyperref}

\usepackage[english]{babel}

\theoremstyle{plain}

\newtheorem*{corollary}{Corollary}
\newtheorem{lemma}{Lemma}
\newtheorem{theorem}{Theorem}
\newtheorem*{conjecture}{Conjecture}

\theoremstyle{remark}
\newtheorem*{remark}{Remark}

\theoremstyle{definition}

\newtheorem{example}{Example}

\DeclareMathOperator{\Id}{Id}

\DeclareMathOperator{\id}{id}

\DeclareMathOperator{\End}{End}
\DeclareMathOperator{\Aut}{Aut}

\DeclareMathOperator{\length}{length}
\DeclareMathOperator{\sign}{sign}
\DeclareMathOperator{\Alt}{Alt}

\DeclareMathOperator{\Hom}{Hom}
\DeclareMathOperator{\PIexp}{PIexp}

\begin{document}

\title[Asymptotics of $H$-identities]{Asymptotics of $H$-identities for associative algebras with an $H$-invariant radical}

\author{A.\,S.~Gordienko}
\address{Vrije Universiteit Brussel, Belgium}
\email{alexey.gordienko@vub.ac.be} 


\begin{abstract}
We prove the existence of the Hopf PI-exponent
 for finite dimensional associative algebras $A$
 with a generalized Hopf action of an associative algebra $H$ with $1$ over an algebraically closed field of characteristic $0$ assuming only the invariance of the Jacobson radical $J(A)$ under the $H$-action
 and the existence of the decomposition of $A/J(A)$ into the sum of $H$-simple algebras.
As a consequence, we show that the analog of Amitsur's conjecture holds for $G$-codimensions of finite dimensional associative algebras over a field of characteristic $0$ with an action of an arbitrary group $G$ by automorphisms and anti-automorphisms and for differential codimensions of finite dimensional
associative algebras with an action of an arbitrary Lie algebra by derivations.
\end{abstract}


\thanks{Supported by Fonds Wetenschappelijk Onderzoek~--- Vlaanderen Pegasus Marie Curie post doctoral fellowship (Belgium) and RFBR grant 13-01-00234a (Russia).}

\maketitle

\section{Introduction}
Amitsur's conjecture on asymptotic behaviour of codimensions of ordinary polynomial
identities was proved by A.~Giambruno and M.\,V.~Zaicev~\cite[Theorem~6.5.2]{ZaiGia} in 1999.

When an algebra is endowed with a grading, an action of a group $G$ by automorphisms and anti-automorphisms, an action of a Lie algebra by derivations or a structure of an $H$-module algebra for some Hopf algebra $H$,
it is natural to consider, respectively, graded, $G$-, differential or $H$-identities~\cite{BahtGiaZai, BahtZaiGradedExp, BereleHopf, Kharchenko}.
The analog of Amitsur's conjecture for finite dimensional associative algebras with a $\mathbb Z_2$-action was proved by A. Giambruno and M.V. Zaicev~\cite[Theorem~10.8.4]{ZaiGia} in 1999.
 In 2010--2011, E.~Aljadeff, A.~Giambruno, and D.~La~Mattina~\cite{AljaGia, AljaGiaLa, GiaLa}
obtained the validity of the analog of Amitsur's conjecture for associative PI-algebras with an action of a finite Abelian group by automorphisms as a particular case of their result for graded algebras.

In 2012, the analog of the conjecture
was proved~\cite{ASGordienko3, GordienkoKochetov} for finite dimensional associative algebras with a rational action of a reductive affine algebraic group by automorphisms and anti-automorphisms,
with an action of a finite dimensional semisimple Lie algebra by derivations or an action
of a semisimple Hopf algebra. These results were obtained as a consequence of~\cite[Theorem~5]{ASGordienko3}
and~\cite[Theorem~6]{GordienkoKochetov}, where the authors considered finite dimensional associative algebras with a generalized Hopf action of an associative algebra $H$ with $1$.
In the proof, they required  the existence of an $H$-invariant Wedderburn~--- Mal'cev and Wedderburn~--- Artin decompositions. Here we remove the first restriction. This enables us to prove
the analog of Amitsur's conjecture for $G$-codimensions of finite dimensional associative algebras with an action of an arbitrary group $G$ by automorphisms and anti-automorphisms and for differential codimensions of finite dimensional associative algebras with an action of an arbitrary Lie algebra by derivations.

\section{Polynomial $H$-identities and their codimensions}\label{SectionH}

Let $H$ be a Hopf algebra over a field $F$. An algebra $A$
over $F$
is an \textit{$H$-module algebra}
or an \textit{algebra with an $H$-action},
if $A$ is endowed with a homomorphism $H \to \End_F(A)$ such that
$h(ab)=(h_{(1)}a)(h_{(2)}b)$
for all $h \in H$, $a,b \in A$. Here we use Sweedler's notation
$\Delta h = h_{(1)} \otimes h_{(2)}$ where $\Delta$ is the comultiplication
in $H$.

In order to embrace an action of a group by anti-automorphisms, we consider a
generalized Hopf action~\cite[Section~3]{BereleHopf}.

Let $H$ be an associative algebra with $1$ over $F$.
We say that an associative algebra $A$ is an algebra with a \textit{generalized $H$-action}
if $A$ is endowed with a homomorphism $H \to \End_F(A)$
and for every $h \in H$ there exist $k\in \mathbb N$ and $h'_i, h''_i, h'''_i, h''''_i \in H$, $1\leqslant i \leqslant k$,
such that 
\begin{equation}\label{EqGeneralizedHopf}
h(ab)=\sum_{i=1}^k\bigl((h'_i a)(h''_i b) + (h'''_i b)(h''''_i a)\bigr) \text{ for all } a,b \in A.
\end{equation}

Choose a basis $(\gamma_\beta)_{\beta \in \Lambda}$ in $H$ and
denote by $F \langle X | H \rangle$ 
the free associative algebra over $F$ with free formal
 generators $x_i^{\gamma_\beta}$, $\beta \in \Lambda$, $i \in \mathbb N$.
 Let $x_i^h := \sum_{\beta \in \Lambda} \alpha_\beta x_i^{\gamma_\beta}$
 for $h= \sum_{\beta \in \Lambda} \alpha_\beta \gamma_\beta$, $\alpha_\beta \in F$,
 where only finite number of $\alpha_\beta$ are nonzero.
Here $X := \lbrace x_1, x_2, x_3, \ldots \rbrace$, $x_j := x_j^1$, $1 \in H$.
 We refer to the elements
 of $F\langle X | H \rangle$ as $H$-polynomials.
Note that here we do not consider any $H$-action on $F \langle X | H \rangle$.

Let $A$ be an associative algebra with a generalized $H$-action.
Any map $\psi \colon X \to A$ has the unique homomorphic extension $\bar\psi
\colon F \langle X | H \rangle \to A$ such that $\bar\psi(x_i^h)=h\psi(x_i)$
for all $i \in \mathbb N$ and $h \in H$.
 An $H$-polynomial
 $f \in F\langle X | H \rangle$
 is an \textit{$H$-identity} of $A$ if $\bar\psi(f)=0$
for all maps $\psi \colon X \to A$. In other words, $f(x_1, x_2, \ldots, x_n)$
 is an $H$-identity of $A$
if and only if $f(a_1, a_2, \ldots, a_n)=0$ for any $a_i \in A$.
 In this case we write $f \equiv 0$.
The set $\Id^{H}(A)$ of all $H$-identities
of $A$ is an ideal of $F\langle X | H \rangle$.
Note that our definition of $F\langle X | H \rangle$
depends on the choice of the basis $(\gamma_\beta)_{\beta \in \Lambda}$ in $H$.
However such algebras can be identified in the natural way,
and $\Id^H(A)$ is the same.

Denote by $P^H_n$ the space of all multilinear $H$-polynomials
in $x_1, \ldots, x_n$, $n\in\mathbb N$, i.e.
$$P^{H}_n = \langle x^{h_1}_{\sigma(1)}
x^{h_2}_{\sigma(2)}\ldots x^{h_n}_{\sigma(n)}
\mid h_i \in H, \sigma\in S_n \rangle_F \subset F \langle X | H \rangle.$$
Then the number $c^H_n(A):=\dim\left(\frac{P^H_n}{P^H_n \cap \Id^H(A)}\right)$
is called the $n$th \textit{codimension of polynomial $H$-identities}
or the $n$th \textit{$H$-codimension} of $A$.

The analog of Amitsur's conjecture for $H$-codimensions of $A$ can be formulated
as follows.

\begin{conjecture}  There exists
 $\PIexp^H(A):=\lim\limits_{n\to\infty}
 \sqrt[n]{c^H_n(A)} \in \mathbb Z_+$.
\end{conjecture}

We call $\PIexp^H(A)$ the \textit{Hopf PI-exponent} of $A$.

\begin{example} Every algebra $A$ is an $H$-module algebra
for $H=F$. In this case the $H$-action is trivial and we get ordinary polynomial identities and their codimensions. 
\end{example}

\begin{example}\label{ExampleIdFG*}
Let $A$ be an associative algebra with an action 
of a group $G$ by automorphisms and anti-automorphisms. Then $A$ is an algebra with
 a generalized $H$-action where $H=FG$.
We introduce the \textit{free $G$-algebra} $F\langle X | G \rangle := F\langle X | H \rangle$,
 the ideal of polynomial $G$-identities $\Id^G(A):=\Id^H(A)$, 
 and $G$-codimensions $c^G_n(A):=c^H_n(A)$.
\end{example}

\begin{example} If $H=U(\mathfrak g)$ where $U(\mathfrak g)$ is the universal enveloping
algebra of a Lie algebra $\mathfrak g$, then an $H$-module algebra is an algebra
with a $\mathfrak g$-action by derivations. The corresponding $H$-identities are called
\textit{differential identities} or \textit{polynomial identities with derivations}.
\end{example}

\begin{theorem}\label{TheoremMainHAssoc} Let $A$ be a finite dimensional non-nilpotent
associative algebra with a generalized $H$-action
where $H$ is an associative algebra with $1$ over an algebraically closed field $F$ of characteristic $0$.
Suppose that the Jacobson radical $J:=J(A)$ is an $H$-submodule.
Let $$A/J = B_1 \oplus \ldots \oplus B_q \text{ (direct sum of $H$-invariant ideals)}$$
where $B_i$ are $H$-simple algebras and let $\varkappa \colon A/J \to A$
be any homomorphism of algebras (not necessarily $H$-linear) such that $\pi\varkappa = \id_{A/J}$ where $\pi \colon A \to A/J$ is the natural projection.
Then there exist constants $C_1, C_2 > 0$, $r_1, r_2 \in \mathbb R$ such that $$C_1 n^{r_1} d^n \leqslant c^{H}_n(A) \leqslant C_2 n^{r_2} d^n\text{ for all }n \in \mathbb N$$
where $$d= \max\dim\left( B_{i_1}\oplus B_{i_2} \oplus \ldots \oplus B_{i_r}
 \mathbin{\Bigl|}  r \geqslant 1,\right.$$ \begin{equation}\label{EqdAssoc}\left. (H\varkappa(B_{i_1}))A^+ \,(H\varkappa(B_{i_2})) A^+ \ldots (H\varkappa(B_{i_{r-1}})) A^+\,(H\varkappa(B_{i_r}))\ne 0\right)\end{equation} and
 $A^+:=A+F\cdot 1$.
\end{theorem}
\begin{remark}
If $A$ is nilpotent, i.e. $x_1 x_2 \ldots x_p\equiv 0$ for some $p\in\mathbb N$, then
$P^{H}_n \subseteq \Id^{H}(A)$ and $c^H_n(A)=0$ for all $n \geqslant p$.
\end{remark}
\begin{corollary}
The analog of Amitsur's conjecture holds
 for such codimensions.
\end{corollary}
\begin{remark}
The existence of the map $\varkappa$ follows from the ordinary Wedderburn~--- Mal'cev theorem.
\end{remark}
Theorem~\ref{TheoremMainHAssoc} will be proved in Sections~\ref{SectionSnCocharUpper} and~\ref{SectionLowerAssoc}.

\section{Applications}
 Here we list some important corollaries from Theorem~\ref{TheoremMainHAssoc}.
 
\begin{theorem}\label{TheoremHmoduleAssoc}
Let $A$ be a finite dimensional non-nilpotent associative $H$-module algebra for a Hopf algebra $H$
over a field $F$ of characteristic $0$. Suppose that the antipode of $H$ is bijective and the Jacobson radical $J(A)$ is an $H$-submodule.
Then there exist constants $d\in\mathbb N$, $C_1, C_2 > 0$, $r_1, r_2 \in \mathbb R$ such that $$C_1 n^{r_1} d^n \leqslant c^{H}_n(A) \leqslant C_2 n^{r_2} d^n\text{ for all }n \in \mathbb N.$$
\end{theorem}
\begin{proof}
Let $K \supset F$ be an extension of the field $F$.
Then $$(A \otimes_F K)/(J(A) \otimes_F K) \cong (A/J(A)) \otimes_F K$$
is again a semisimple algebra. Since $J(A) \otimes_F K$ is nilpotent, $J(A \otimes_F K) = J(A) \otimes_F K$.
In particular, $J(A \otimes_F K)$ is $H\otimes_F K$-invariant.

Moreover, $H$-codimensions do not change upon an extension of the base field.
The proof is analogous to the case of ordinary codimensions~\cite[Theorem~4.1.9]{ZaiGia}.
Hence we may assume $F$ to be algebraically closed.
By~\cite[Lemma~1]{GordienkoKochetov}, $A/J(A) = B_1 \oplus \ldots \oplus B_q$ (direct sum of $H$-invariant ideals) for some $H$-simple algebras $B_i$. Now we apply Theorem~\ref{TheoremMainHAssoc}.
\end{proof}
\begin{theorem}\label{TheoremDiffAssoc}
Let $A$ be a finite dimensional non-nilpotent
associative algebra  over a field $F$ of characteristic $0$ with an action of a Lie algebra $\mathfrak g$ by derivations. Then there exist constants $d\in\mathbb N$, $C_1, C_2 > 0$, $r_1, r_2 \in \mathbb R$ such that $$C_1 n^{r_1} d^n \leqslant c^{U(\mathfrak g)}_n(A) \leqslant C_2 n^{r_2} d^n\text{ for all }n \in \mathbb N.$$
\end{theorem}
\begin{proof}
By \cite[Lemma 3.2.2]{Dixmier}, the Jacobson radical (which coincides
with the prime radical) of a finite dimensional associative algebra is invariant under all
derivations. Hence we may apply Theorem~\ref{TheoremHmoduleAssoc}.
\end{proof}

\begin{theorem}\label{TheoremGAssoc}
Let $A$ be a finite dimensional non-nilpotent
associative algebra over a field $F$ of characteristic $0$ with an action of a group $G$ by automorphisms and anti-automorphisms.
Then there exist constants $d\in\mathbb N$, $C_1, C_2 > 0$, $r_1, r_2 \in \mathbb R$ such that $$C_1 n^{r_1} d^n \leqslant c^{G}_n(A) \leqslant C_2 n^{r_2} d^n\text{ for all }n \in \mathbb N.$$
\end{theorem}
\begin{proof}
Again, $G$-codimensions do not change upon an extension of the base field.
Hence we may assume $F$ to be algebraically closed. The radical is invariant under all automorphisms and anti-automorphisms. Now we apply~\cite[Lemma~2]{GordienkoKochetov} and Theorem~\ref{TheoremMainHAssoc}.
\end{proof}

The algebra in the example below has no $G$-invariant Wedderburn~--- Mal'cev decomposition, however it satisfies the analog of Amitsur's conjecture.

\begin{example}[Yuri Bahturin]\label{ExampleGnoninvWedderburn}
Let $F$ be a field of characteristic $0$ and let $$A = \left\lbrace\left(\begin{array}{cc} C & D \\
0 & 0
  \end{array}\right) \mathrel{\biggl|} C, D\in M_m(F)\right\rbrace
  \subseteq M_{2m}(F),$$ $m \geqslant 2$.
      Consider $\varphi \in \Aut(A)$ where
  $$\varphi\left(\begin{array}{cc} C & D \\
0 & 0
  \end{array}\right)=\left(\begin{array}{cc} C & C+D \\
0 & 0
  \end{array}\right).$$
  Then $A$ is an algebra with an action of the group $G=\langle \varphi \rangle
  \cong \mathbb Z$ by automorphisms. There is no $G$-invariant Wedderburn~--- Mal'cev
  decomposition for $A$, however there exist constants $C_1, C_2 > 0$, $r_1, r_2 \in \mathbb R$ such that $$C_1 n^{r_1} m^{2n} \leqslant c^G_n(A) \leqslant C_2 n^{r_2} m^{2n}\text{ for all }n \in \mathbb N.$$
\end{example}
\begin{proof} 
Note that
 \begin{equation}\label{EqGnoninvWedderburn}
  J(A)=\left\lbrace\left(\begin{array}{cc} 0 & D \\
0 & 0
  \end{array}\right) \mathrel{\biggl|} D\in M_m(F)\right\rbrace.
  \end{equation}
Suppose $A = B \oplus J(A)$ (direct sum of $G$-invariant subspaces)
for some maximal semisimple subalgebra $B$. Since $\varphi(a)-a \in J(A)$
for all $a \in A$, we have $\varphi(a)=a$ for all $a \in B$. Thus
$B \subseteq J(A)$ and we get a contradiction. Therefore, there is no $G$-invariant Wedderburn~--- Mal'cev
  decomposition for $A$.

Again, $G$-codimensions do not change upon an extension of the base field.
Moreover, upon an extension of $F$, $A$ 
remains an algebra of the same type.
Thus without loss of generality we may assume
 $F$ to be algebraically closed.
 
 Note that $A/J \cong M_m(F)$
  is a simple algebra. Hence $\PIexp^G(A)=\dim M_m(F)= m^2$
  by Theorems~\ref{TheoremMainHAssoc} and~\ref{TheoremGAssoc}.
\end{proof}

\begin{example}
Let $A$ be the associative algebra from Example~\ref{ExampleGnoninvWedderburn}.
      Denote by $\mathfrak g$ the corresponding Lie algebra with the commutator $[x,y]=xy-yx$
      and consider the adjoint action of $\mathfrak g$ on $A$ by derivations.
        Then there is no $\mathfrak g$-invariant Wedderburn~--- Mal'cev
  decomposition for $A$, however there exist constants $C_1, C_2 > 0$, $r_1, r_2 \in \mathbb R$ such that $$C_1 n^{r_1} m^{2n} \leqslant c^{U(\mathfrak g)}_n(A) \leqslant C_2 n^{r_2} m^{2n}\text{ for all }n \in \mathbb N.$$
\end{example}
\begin{proof} Suppose $A= B \oplus J(A)$ (direct sum of $\mathfrak g$-submodules) for some maximal semisimple associative subalgebra $B$.
  Then $B$ is a Lie ideal of $\mathfrak g$. By~(\ref{EqGnoninvWedderburn}), $J(A)$ is Abelian as a Lie algebra. Thus the center of $\mathfrak g$ contains $J(A)$, which is not true.
  We get a contradiction. Hence there is no $\mathfrak g$-invariant Wedderburn~--- Mal'cev
  decomposition for $A$.
  
  Again, without loss of generality we may assume
 $F$ to be algebraically closed. Since $A/J \cong M_m(F)$
  is a simple algebra, $\PIexp^{U(\mathfrak g)}(A)=\dim M_m(F)= m^2$
  by Theorems~\ref{TheoremMainHAssoc} and~\ref{TheoremDiffAssoc}.
\end{proof}
\begin{remark} The radical of Sweedler's algebra with an action of its dual is not $H$-invariant,
however the analog of Amitsur's conjecture holds for its $H$-identities~\cite[Section~7.4]{ASGordienko3}.
\end{remark}

\section{$S_n$-cocharacters and upper bound}\label{SectionSnCocharUpper}

One of the main tools in the investigation of polynomial
identities is provided by the representation theory of symmetric groups.

Let $A$ be an associative algebra with a generalized $H$-action
where $H$ is an associative algebra with $1$ over a field $F$ of characteristic $0$.
 The symmetric group $S_n$  acts
 on the spaces $\frac {P^H_n}{P^H_{n}
  \cap \Id^H(A)}$
  by permuting the variables.
  Irreducible $FS_n$-modules are described by partitions
  $\lambda=(\lambda_1, \ldots, \lambda_s)\vdash n$ and their
  Young diagrams $D_\lambda$.
   The character $\chi^H_n(A)$ of the
  $FS_n$-module $\frac {P^H_n}{P^H_n
   \cap \Id^H(A)}$ is
   called the $n$th
  \textit{cocharacter} of polynomial $H$-identities of $A$.
   We can rewrite $\chi^H_n(A)$ as
  a sum $$\chi^H_n(A)=\sum_{\lambda \vdash n}
   m(A, H, \lambda)\chi(\lambda)$$ of
  irreducible characters $\chi(\lambda)$.
Let  $e_{T_{\lambda}}=a_{T_{\lambda}} b_{T_{\lambda}}$
and
$e^{*}_{T_{\lambda}}=b_{T_{\lambda}} a_{T_{\lambda}}$
where
$a_{T_{\lambda}} = \sum_{\pi \in R_{T_\lambda}} \pi$
and
$b_{T_{\lambda}} = \sum_{\sigma \in C_{T_\lambda}}
 (\sign \sigma) \sigma$,
be Young symmetrizers corresponding to a Young tableau~$T_\lambda$.
Then $M(\lambda) = FS_n e_{T_\lambda} \cong FS_n e^{*}_{T_\lambda}$
is an irreducible $FS_n$-module corresponding to
 a partition~$\lambda \vdash n$.
  We refer the reader to~\cite{Bahturin, DrenKurs, ZaiGia}
   for an account
  of $S_n$-representations and their applications to polynomial
  identities.

Theorem~\ref{TheoremAssocColength} below is a generalization
of~\cite[Theorem~13~(b)]{BereleHopf} and the remark after~\cite[Theorem~14]{BereleHopf}.  
  
\begin{theorem}\label{TheoremAssocColength}
Let $A$ be a finite dimensional associative algebra with a generalized $H$-action
where $H$ is an associative algebra with $1$ over a field $F$ of characteristic $0$.
Then there exist constants $C_3 > 0$, $r_3 \in \mathbb N$ such that
 $$\sum_{\lambda \vdash n} m(A,H,\lambda)
\leqslant C_3n^{r_3}\text{ for all }n \in \mathbb N.$$
\end{theorem}
\begin{remark}
Cocharacters do not change
upon an extension of the base field $F$
(the proof is completely analogous to \cite[Theorem 4.1.9]{ZaiGia}),
 so we may assume $F$ to be algebraically closed.
\end{remark}
\begin{proof}
Consider ordinary polynomial identities and cocharacters of $A$.
We may define them as $H$-identities and $H$-cocharacters
for $H=F$:
$P_n := P^F_n$,
$\chi_n(A) := \chi^F_n(A)$,
$m(A,\lambda):= m(A,F,\lambda)$,
$\Id(A):=\Id^F(A)$.
By the Berele~--- Regev theorem (see e.g.~\cite[Theorem~4.9.3]{ZaiGia}),
\begin{equation}\label{EqAssocOrdinaryColength}
\sum_{\lambda \vdash n} m(A,\lambda)
\leqslant C_4 n^{r_4}
\end{equation}
 for some $C_4 > 0$ and $r_4 \in \mathbb N$.

Let $G_1 \subseteq G_2$ be finite groups and let $W$ be
an $FG_2$-module.
Denote by $W \downarrow G_1$ the module $W$ with the $G_2$-action restricted to $G_1$.

Let $\zeta \colon H \to \End_F(A)$  be the homomorphism corresponding to the $H$-action,
and let $\bigl(\zeta(\gamma_j)\bigr)_{j=1}^m$, $\gamma_j \in H$, be a basis in $\zeta(H)$. 

Consider the diagonal embedding $\varphi \colon S_n \to S_{mn}$,
$$\varphi(\sigma):=\left(\begin{array}{cccc|lllc|c}
1 & 2 & \ldots & n & n+1 & n+2 & \ldots & 2n & \ldots  \\
\sigma(1) & \sigma(2) & \ldots & \sigma(n) &
n+\sigma(1) & n+\sigma(2) & \ldots & n+\sigma(n) & \ldots
\end{array}
\right)$$
and the $S_n$-homomorphism $\pi \colon (P_{mn}\downarrow \varphi(S_n)) \to P_n^H$
 defined by $\pi(x_{n(i-1)+t})=x^{\gamma_i}_t$, $1 \leqslant i \leqslant m$,
 $1 \leqslant t \leqslant n$. Note that $\pi(P_{mn} \cap \Id(A))
 \subseteq P^H_n \cap \Id^H(A)$ and 
 $x^h - \sum_{j=1}^m \alpha_j x^{\gamma_j} \in \Id^H(A)$
 for all $h \in H$ and $\alpha_j \in F$ such that $\zeta(h)=\sum_{j=1}^m \alpha_j \zeta(\gamma_j)$.
 Hence the $FS_n$-module $\frac{P^H_n}{P^H_n \cap \Id^H(A)}$
 is a homomorphic image of the $FS_n$-module
 $\left(\frac{P_{mn}}{P_{mn} \cap \Id(A)}\right)\downarrow \varphi(S_n)$.
 Denote by  $\length(M)$ the number of irreducible components
of a module $M$.
Then
$$ \sum_{\lambda \vdash n} m(A,H,\lambda)
= \length\left(\frac{P^H_n}{P^H_n \cap \Id^H(A)}\right)
 \leqslant \length\left(\left(\frac{P_{mn}}{P_{mn} \cap \Id(A)}\right)\downarrow \varphi(S_n)\right).
$$

Therefore, it is sufficient to prove
 that  $\length\left(\left(\frac{P_{mn}}{P_{mn} \cap \Id(A)}\right)\downarrow \varphi(S_n)
 \right)$ is polynomially bounded.
Replacing $|G|$ with $m$ in~\cite[Lemma 10 and 12]{ASGordienko2} (or, alternatively, using the proof of~\cite[Theorem~13~(b)]{BereleHopf}), we derive this from~(\ref{EqAssocOrdinaryColength})
and~\cite[Theorem~4.6.2]{ZaiGia}.
\end{proof}

In the next two lemmas we consider
 a finite dimensional associative algebra with a generalized $H$-action
 having an $H$-invariant nilpotent ideal $J$
where $H$ is an associative algebra with $1$ over a field $F$ of characteristic $0$ and $J^p=0$
for some $p\in\mathbb N$.
Fix a decomposition $A/J = B_1 \oplus \ldots \oplus B_q $
where $B_i$ are some subspaces.
 Let $\varkappa \colon A/J \to A$
be an $F$-linear map such that $\pi\varkappa = \id_{A/J}$ where $\pi \colon A \to A/J$
is the natural projection. Define the number $d$ by~(\ref{EqdAssoc}).

\begin{lemma}\label{LemmaAssocUpperCochar}
 Let $n\in\mathbb N$ and $\lambda = (\lambda_1, \ldots, \lambda_s) \vdash n$. Then if $\sum_{k=d+1}^s \lambda_k \geqslant p$, we have $m(A, H, \lambda)=0$. 
\end{lemma}
\begin{proof}
It is sufficient to prove that $e^{*}_{T_\lambda}f \in \Id^H(A)$ for all $f \in P_n$ and for all Young tableaux $T_\lambda$ corresponding to $\lambda$.

Fix a basis in $A$ that is a union of bases of~$\varkappa(B_1),\ldots, \varkappa(B_q)$ and~$J$.
Since $e^{*}_{T_\lambda}f$ is multilinear, it sufficient to prove that $e^{*}_{T_\lambda}f$
vanishes under all evaluations on basis elements.
Fix some substitution of basis elements and choose $1 \leqslant i_1,\ldots,i_r \leqslant q$
such that all the elements substituted belong to $\varkappa(B_{i_1})\oplus \ldots \oplus \varkappa(B_{i_r}) \oplus J$, and for each $k$ we have an element being substituted from $\varkappa(B_{i_k})$.
Then we may assume that $\dim(B_{i_1}\oplus \ldots \oplus B_{i_r}) \leqslant d$,
since otherwise $e^{*}_{T_\lambda}f$ is zero by the definition of $d$.
 Note that
$e^{*}_{T_\lambda} = b_{T_\lambda} a_{T_\lambda}$
and $b_{T_\lambda}$ alternates the variables of each column
of $T_\lambda$. Hence if $e^{*}_{T_\lambda} f$ does not vanish, this implies that different basis elements
are substituted for the variables of each column. 
Therefore, at least $\sum_{k=d+1}^s \lambda_k \geqslant p$ elements must be taken from $J$.
Since $J^p = 0$, we have $e^{*}_{T_\lambda} f \in \Id^H(A)$.
\end{proof}

\begin{lemma}\label{LemmaAssocUpper} 
If $d > 0$, then there exist constants $C_2 > 0$, $r_2 \in \mathbb R$
such that $c^H_n(A) \leqslant C_2 n^{r_2} d^n$
for all $n \in \mathbb N$. In the case $d=0$, the algebra $A$ is nilpotent.
\end{lemma}
\begin{proof}
Lemma~\ref{LemmaAssocUpperCochar} and~\cite[Lemmas~6.2.4, 6.2.5]{ZaiGia}
imply
$$
\sum_{m(A,H, \lambda)\ne 0} \dim M(\lambda) \leqslant C_3 n^{r_3} d^n
$$
for some constants $C_3, r_3 > 0$.
Together with Theorem~\ref{TheoremAssocColength} this inequality yields the upper bound.
\end{proof}

\section{Lower bound}\label{SectionLowerAssoc}

As usual, in order to prove the lower bound, it is sufficient to provide
a polynomial alternating on sufficiently many sufficiently large sets of variables.
Lemma~\ref{LemmaAssocLowerPolynomial} below is a generalization of~\cite[Lemma~10]{ASGordienko3}.

\begin{lemma}\label{LemmaAssocLowerPolynomial}
Let $A$, $J$, $\varkappa$, $B_i$, and $d$ be the same as in Theorem~\ref{TheoremMainHAssoc}.
If $d > 0$, then there exists a number $n_0 \in \mathbb N$ such that for every $n\geqslant n_0$
there exist disjoint subsets $X_1$, \ldots, $X_{2k} \subseteq \lbrace x_1, \ldots, x_n
\rbrace$, $k := \left[\frac{n-n_0}{2d}\right]$,
$|X_1| = \ldots = |X_{2k}|=d$ and a polynomial $f \in P^H_n \backslash
\Id^H(A)$ alternating in the variables of each set $X_j$.
\end{lemma}
\begin{proof} 
Without loss of generality,
we may assume that $d = \dim(B_1 \oplus B_2 \oplus \ldots \oplus B_r)$
where 
$(H\varkappa(B_1))A^+ (H\varkappa(B_2))A^+ \ldots (H\varkappa(B_{r-1}))A^+ (H\varkappa(B_r))\ne 0$.

 Since $J$ is nilpotent, we can find maximal $\sum_{i=1}^r q_i$, $q_i \in \mathbb Z_+$,
 such that
$$\left(a_1 \prod_{i=1}^{q_1} j_{1i}\right) {\gamma_1}\varkappa(b_1)
\left(a_2 \prod_{i=1}^{q_2} j_{2i}\right)
 {\gamma_2}\varkappa(b_2) \ldots \left(a_r \prod_{i=1}^{q_r} j_{ri}\right)
 {\gamma_r}\varkappa(b_r)  \left(a_{r+1} \prod_{i=1}^{q_{r+1}} j_{r+1,i}\right) \ne 0$$ for
  some $j_{ik}\in J$, $a_i \in A^+$, $b_i \in B_i$, $\gamma_i \in H$.
 Let $j_i := a_i\prod_{k=1}^{q_i} j_{ik}$.
 
 Then \begin{equation}\label{EqAssocNonZero}j_1 {\gamma_1}\varkappa(b_1)
j_2  {\gamma_2}\varkappa(b_2) \ldots j_r  {\gamma_r}\varkappa(b_r)j_{r+1} \ne 0\end{equation}
 for some $b_i \in B_i$, $\gamma_i \in H,$
 however
 \begin{equation}\label{EqAssocbazero}j_1
  \tilde b_1 j_2 \tilde b_2
  \ldots j_r \tilde b_r j_{r+1} = 0 \end{equation} for all $\tilde b_i\in A^+(H\varkappa(B_i))A^+$  such that $\tilde b_k\in J(H\varkappa(B_k))A^+ + A^+(H\varkappa(B_k))J$ for at least one $k$.
 
 Recall that $\varkappa$ is a homomorphism of algebras.
Moreover $\pi(h\varkappa(a)-\varkappa(ha))=0$ 
implies $h\varkappa(a)-\varkappa(ha) \in J$ for all $a\in A$ and $h\in H$.
Hence, by~(\ref{EqAssocbazero}), if we replace $\varkappa(b_i)$ in 
the left-hand side of~(\ref{EqAssocNonZero})
with a product of $\varkappa(b_i)$ and an expression involving $\varkappa$, the map $\varkappa$
 will behave like a homomorphism of $H$-modules.
 We will exploit this property further.

Let $a^{(i)}_{k}$, $1 \leqslant k \leqslant d_k := \dim B_i$,
 be a basis in $B_i$, $1 \leqslant k \leqslant r$.
 
 In virtue of~\cite[Theorem~7]{ASGordienko3},
there exist constants $m_t \in \mathbb Z_+$
such that for any $k$ there exist
 multilinear polynomials $$f_t=f_t(x^{(t, 1)}_1,
 \ldots, x^{(t, 1)}_{d_t};
 \ldots;  x^{(t, 2k)}_1,
 \ldots, x^{(t, 2k)}_{d_t}; z^{(t)}_1, \ldots, z^{(t)}_{m_t}; z_t) \in P^H_{2k d_t+m_t+1}$$ 
alternating in the variables from disjoint sets
$X^{(t)}_{\ell}=\lbrace x^{(t, \ell)}_1, x^{(t, \ell)}_2,
\ldots, x^{(t, \ell)}_{d_t} \rbrace$, $1 \leqslant \ell \leqslant 2k$
and elements $\bar z^{(t)}_\alpha \in B_{t}$, $1 \leqslant \alpha \leqslant m_t$,
such that
$$f_t(a^{(t)}_1,
 \ldots, a^{(t)}_{d_t};
 \ldots;  a^{(t)}_1,
 \ldots, a^{(t)}_{d_t}; \bar z^{(t)}_1, \ldots, \bar z^{(t)}_{m_t}; \bar z_t)=\bar z_t$$
 for any $\bar z_t \in B_t$. (In~\cite[Theorem~7]{ASGordienko3} the author requires $\dim H < +\infty$, but one can replace
$H$ with its image $\tilde H$ in $\Hom_F(A)$ and notice that $\dim \tilde H < +\infty$ and $c_n^H(A) = c_n^{\tilde H}(A)$ for all $n\in\mathbb N$.)

 Let $n_0 = 2r+1+\sum_{i=1}^r m_i$, $k = \left[\frac{n-n_0}{2d}\right]$, $\tilde k =
 \left[\frac{(n-2kd-n_0)-m_1}{2d_1}\right]+1$. We choose $f_t$ for  $B_t$ and $k$,
 $1 \leqslant t \leqslant r$. In addition, again by~\cite[Theorem~7]{ASGordienko3},
  we take $\tilde f_1$ for $B_1$ and $\tilde k$.
  Let $$\hat f_0 := v_1 \gamma_1\varkappa \left(f_1\left(x^{(1, 1)}_1,
 \ldots, x^{(1, 1)}_{d_1};
 \ldots;  x^{(1, 2k)}_1,
 \ldots, x^{(1, 2k)}_{d_1}; z^{(1)}_1, \ldots, z^{(1)}_{m_1}; \right.\right.$$ $$\left.\left.
 \tilde f_1(y^{(1)}_1,
 \ldots, y^{(1)}_{d_1};
 \ldots;  y^{(2\tilde k)}_1,
 \ldots, y^{(2\tilde k)}_{d_1}; u_1, \ldots, u_{m_1}; z_1)\right)\right) v_2\cdot $$ $$
 \prod_{t=2}^{r}\left(\gamma_t\varkappa\left(  f_t(x^{(t, 1)}_1,
 \ldots, x^{(t, 1)}_{d_t};
 \ldots;  x^{(t, 2k)}_1,
 \ldots, x^{(t, 2k)}_{d_t}; z^{(t)}_1, \ldots, z^{(t)}_{m_t}; z_t)\right) v_{t+1}\right).$$
 
 The value of the multilinear function $\hat f_0$ under the substitution $x^{(t, \alpha)}_{\beta}=a^{(t)}_\beta$,
 $z^{(t)}_{\beta}=\bar z^{(t)}_\beta$, $v_t=j_t$, $z_t=b_t$,
  $y^{(\alpha)}_{\beta}=a^{(1)}_\beta$, $u_{\beta}=\bar z^{(1)}_\beta$
  equals the left-hand side of~(\ref{EqAssocNonZero}), which is nonzero.

    Note that $f_i$ are multilinear in $z_i$ and $\tilde f_1$ is multilinear in $z_1$.
    Therefore we may assume that for a fixed $i$ in all entries of $z_i^h$ in $f_i$
    the element $h\in H$ is the same. Analogously, we may assume that in all entries of $z_1^h$ in $
    \tilde f_1$ the element $h\in H$ is the same. Furthermore, we can hide these $h$ inside the
    elements substituted for $z_i$  in the case of $\tilde f_1$ and $f_i$, $i \geqslant 2$, and inside $\tilde f_1$ in the case of $f_1$, and the value $b$ of $\hat f_0$ is still nonzero under
    the substitution
    $x^{(t, \alpha)}_{\beta}=a^{(t)}_\beta$,
 $z^{(t)}_{\beta}=\bar z^{(t)}_\beta$, $v_t=j_t$, $z_t=h_t b_t$,
  $y^{(\alpha)}_{\beta}=a^{(1)}_\beta$, $u_{\beta}=\bar z^{(1)}_\beta$
  for some $h_t \in H$.
  
  As we have mentioned, $\varkappa$ is a homomorphism of algebras and, by~(\ref{EqAssocbazero}), behaves like a homomorphism of $H$-modules. Note that do not need  the last property for $z_i^h$
  since we may assume that in all such entries we have $h=1$.
  Hence the value of 
$$f_0 := v_1  \left(f_1\left(x^{(1, 1)}_1,
 \ldots, x^{(1, 1)}_{d_1};
 \ldots;  x^{(1, 2k)}_1,
 \ldots, x^{(1, 2k)}_{d_1}; z^{(1)}_1, \ldots, z^{(1)}_{m_1}; \right.\right.$$ $$\left.\left.
 \tilde f_1(y^{(1)}_1,
 \ldots, y^{(1)}_{d_1};
 \ldots;  y^{(2\tilde k)}_1,
 \ldots, y^{(2\tilde k)}_{d_1}; u_1, \ldots, u_{m_1}; z_1)\right)\right)^{\gamma_1} v_2\cdot $$ $$
 \prod_{t=2}^{r}\left(\left(  f_t(x^{(t, 1)}_1,
 \ldots, x^{(t, 1)}_{d_t};
 \ldots;  x^{(t, 2k)}_1,
 \ldots, x^{(t, 2k)}_{d_t}; z^{(t)}_1, \ldots, z^{(t)}_{m_t}; z_t)\right)^{\gamma_t} v_{t+1}\right)$$
 under the substitution
    $x^{(t, \alpha)}_{\beta}=\varkappa(a^{(t)}_\beta)$,
 $z^{(t)}_{\beta}=\varkappa(\bar z^{(t)}_\beta)$, $v_t=j_t$, $z_t=\varkappa(h_t b_t)$,
  $y^{(\alpha)}_{\beta}=\varkappa(a^{(1)}_\beta)$, $u_{\beta}=\varkappa(\bar z^{(1)}_\beta)$
  is again $b\ne 0$. 
    We denote this substitution by $\Xi$.

 Note that without additional manipulations a composition
 of $H$-polynomials is only a multilinear function but not an $H$-polynomial.
 However, using~(\ref{EqGeneralizedHopf}), we can always represent
 such function by an $H$-polynomial. Here we make such manipulations at the very end of the proof.
 
Let $X_\ell = \bigcup_{t=1}^r X^{(t)}_\ell$ and let $\Alt_\ell$
be the operator of alternation on the set $X_\ell$.
   Denote $\hat f := \Alt_1 \Alt_2 \ldots \Alt_{2k} f_0$.
   Note that the alternations do not change $z_t$,
   and $f_t$ is alternating on each $X^{(t)}_\ell$.
   Hence the value of $\hat f$ under the substitution $\Xi$
   equals $\left((d_1)! (d_2)! \ldots (d_{r})!\right)^{2k}\ b \ne 0$
   since $\varkappa(B_1) \oplus \ldots \oplus \varkappa(B_r)$ is a direct sum
   of (not necessarily $H$-invariant) ideals and if the alternation puts a variable from
   $X^{(t)}_\ell$ on the place of a variable from $X^{(t')}_\ell$
   for $t \ne t'$, the corresponding $h\varkappa(a^{(t)}_\beta)$, $h\in H$, annihilates $\varkappa(h_{t'}b_{t'})$.

   Note that $\tilde f_1$ is a linear combination
   of multilinear monomials $W$, and one of the terms
   $$\Alt_1 \Alt_2 \ldots \Alt_{2k} v_1 \left(f_1(x^{(1, 1)}_1,
 \ldots, x^{(1, 1)}_{d_1};
 \ldots;  x^{(1, 2k)}_1,
 \ldots, x^{(1, 2k)}_{d_1}; z^{(1)}_1, \ldots, z^{(1)}_{m_1}; W)\right)^{\gamma_1}v_2\cdot $$ $$
 \prod_{q=2}^{r}\left(\left( 
 f_t(x^{(t, 1)}_1,
 \ldots, x^{(t, 1)}_{d_t};
 \ldots;  x^{(t, 2k)}_1,
 \ldots, x^{(t, 2k)}_{d_t}; z^{(t)}_1, \ldots, z^{(t)}_{m_t}; z_t)\right)^{\gamma_q}v_{q+1}\right)$$
 in $\hat f$ does not vanish under the substitution $\Xi$.
 Moreover, $$\deg \hat f = 2kd + (2\tilde k d_1+m_1) + n_0 > n$$ and $\deg W = \deg \tilde f_1
 = 2\tilde k d_1+m_1+1$.
 Let $W=w_1 w_2 \ldots w_{2\tilde k d_1+m_1+1}$ where $w_i$ are  variables
 from the set $\lbrace y^{(1)}_1,
 \ldots, y^{(1)}_{d_1};
 \ldots;  y^{(2\tilde k)}_1,
 \ldots, y^{(2\tilde k)}_{d_1}; u_1, \ldots, u_{m_1}; z_1\rbrace$
 replaced under the substitution $\Xi$ with $\bar w_i \in \varkappa(B_1)$.
 Let $$f:=\Alt_1 \Alt_2 \ldots \Alt_{2k}  v_1 \left(f_1(x^{(1, 1)}_1,
 \ldots, x^{(1, 1)}_{d_1};
 \ldots;  x^{(1, 2k)}_1,
 \ldots, x^{(1, 2k)}_{d_1}; z^{(1)}_1, \ldots, z^{(1)}_{m_1}; \right.$$ $$\left.
 w_1 w_2 \ldots w_{n-2kd-n_0} z)\right)^{\gamma_1}v_2 \cdot $$ $$
 \prod_{q=2}^{r}\left(\left( 
 f_t(x^{(t, 1)}_1,
 \ldots, x^{(t, 1)}_{d_1};
 \ldots;  x^{(t, 2k)}_1,
 \ldots, x^{(t, 2k)}_{d_1}; z^{(t)}_1, \ldots, z^{(t)}_{m_t}; z_t)\right)^{\gamma_q}v_{q+1}\right)$$
where $z$ is an additional variable.
Then using~(\ref{EqGeneralizedHopf}), we may assume $f \in P_n^H$.
Note that $f$ is alternating in $X_\ell$, $1 \leqslant \ell \leqslant 2k$,
and does not vanish under the substitution $\Xi$
with $z = \bar w_{n-2kd-n_0+1} \ldots \bar w_{2\tilde k d_1+m_1+1}$.
 Thus $f$ satisfies all the conditions of the lemma.
 \end{proof}
\begin{proof}[Proof of Theorem~\ref{TheoremMainHAssoc}]
 Now we repeat verbatim the proofs of~\cite[Lemma~11 and Theorem~5]{ASGordienko3}
 using Lemmas~\ref{LemmaAssocUpperCochar} and~\ref{LemmaAssocLowerPolynomial} instead of~\cite[Lemma~10 and Theorem~6]{ASGordienko3}. 
\end{proof}

\section*{Acknowledgements}

This work started while I was an AARMS postdoctoral fellow at Memorial University of Newfoundland, whose faculty and staff I would like to thank for hospitality. I am grateful to Yuri Bahturin, who suggested that I study polynomial $H$-identities, and to Mikhail Zaicev, who suggested that I consider algebras without
an $H$-invariant Wedderburn~--- Mal'cev decomposition.

\end{document}